\documentclass[12pt]{amsart}

\def\ann{\operatorname{ann}}
\def\mod{\operatorname{mod}}
\def\m{\mathfrak{m}}
\def\n{\mathfrak{n}}
\def\pa{\partial}

\usepackage{color}
\definecolor{red}{rgb}{1.00,0.00,0.00}

\newtheorem{theorem}{Theorem}[section]
\newtheorem{lemma}[theorem]{Lemma}
\newtheorem{corollary}[theorem]{Corollary}
\newtheorem{proposition}[theorem]{Proposition}
\newtheorem{example}[theorem]{Example}
\newtheorem{remark}[theorem]{Remark}

\newtheorem{notation}[theorem]{Notation}

\begin{document}

\title[On the number of generators of ideals defining Gorenstein Artin Algebras]{On the number of generators of  ideals defining Gorenstein Artin  
algebras with Hilbert function $ (1,n+1 , 1+{n+1\choose 2}, \ldots,{n+1\choose 2}+1, n+1 ,1)$}

\author{Sabine El Khoury}
\address{Department of Mathematics, American University of Beirut,
Beirut, Lebanon}
\email{se24@aub.edu.lb}
\author{A. V. Jayanthan}
\address{Department of Mathematics, I.I.T. Madras, Chennai, Indian -
600036.}
\email{jayanav@iitm.ac.in}
\author{Hema Srinivasan}
\address{Department of Mathematics, University of Missouri, Columbia,
Missouri, USA}
\email{srinivasanh@missouri.edu}

\begin{abstract}
Let $R = k[w, x_1, \dots, x_n]/I$  be a graded Gorenstein Artin
algebra .  Then $I  = \ann F$ for some $F$ in the divided power
algebra $k_{DP}[W, X_1,\ldots, X_n]$.  If $RI_2$ is a height one ideal
generated by $n$ quadrics, then $I_2 \subset (w)$ after a possible
change of variables.  Let $J = I \cap k[x_1, \dots , x_n]$. Then $\mu(I) \le \mu(J)+n+1$ and $I$
is said to be generic if $\mu(I) = \mu(J) + n+1$. In this article we
prove necessary conditions, in terms of $F$, for an ideal to be
generic.  With some extra assumptions on the exponents of terms of
$F$, we obtain a characterization for $I = \ann F$ to be generic in
codimension four. 
\end{abstract}

\maketitle

\section*{Introduction}
Let $R = k[x_1,\ldots, x_r]$ and $I$ be a homogeneous ideal of $R$.
Let $A = R/I$ be an Artinian Gorenstein quotient of $R$. If $r = 2$,
then it is known that $I$ is a complete intersection. If $r = 3$, then by
Buchsbaum and Eisenbud structure theorem \cite{BE},  $I$ is the
$(2n)^{ {th}}$ order pfaffians of a $(2n+1)^{{th}}$ order skew
symmetric matrix. When $r = 4$, Kustin and Miller gave a structure
theorem for Gorenstein Artinian ideals of the form $(f,g,h, x_4J)$,
where $J$ is a Gorenstein ideal of height three. Let $I =
\oplus_{n\geq 1}I_n$ be the direct sum decomposition of a Gorenstein
Artinian ideal in $k[x,y,z,w]$.  In \cite{is}, Iarrobino and
Srinivasan studied several properties of the Gorenstein ideals $I$
such that $I_2 \cong \langle wx, wy, wz \rangle$ or $I_2 \cong \langle
wx, wy, w^2 \rangle$. They gave a structure theorem for ideals $I$
with $I_2 = \langle wx, wy, wz \rangle$ and Hilbert function $H_{R/I}
= (1,4,7, \ldots, 1)$. They did this by establishing a connection
between some properties of the ideal $I$ and $J = I \cap R',$ where
$R' = k[x,y,z]$, which is a height three Gorenstein ideal.  When $I_2
= \langle wx, wy, w^2 \rangle$, they named these algebras
\textit{mysterious Gorenstein algebras} and studied their various
properties. They showed that such an ideal can be obtained as an
annihilator of a homogeneous form $F = G(X,Y,Z) + WZ^{[j-1]} \in
k_{DP}[X,Y,Z,W]$, where $k_{DP}[X,Y,Z,W]$ denotes the divided power
algebra. They studied the Hilbert function properties of $R/I$
connecting it with those of $R'/J$.

In \cite{E-S}, El Khoury and Srinivasan studied certain properties of
Gorenstein Artinian algebras of the form $R/I$, where $I_2 = \langle wx,
wy, w^2 \rangle$. They gave a structure theorem for such ideals. They
showed that $I$ is generated by $I_2$, $J = I \cap R'$ and an
element of the form $wz^\beta - g$, where $g \in R' \notin I$.

It can easily be seen that $gx, gy \in J$. El Khoury and Srinivasan
proved that  \\ $\left\{ wx, wy, w^2, \alpha_1, \ldots, \alpha_{n-1},
z^j, wz^\beta - g \right\}$ is a minimal generating set for $I$ unless
$gx$ or $gy$ is already a minimal generator of $J = (\alpha_1,
\ldots, \alpha_{n-1}, z^j)$. In that case, dropping each of $gx$
and $gy$ that is a minimal generator for $J$, the remaining $n+3$ or
$n+2$ elements minimally generate $I$. They obtained a minimal free
resolution of $R/I$ in all three cases where neither $gx$ nor $gy$
is a part of a minimal generating set for $J$, either one of them is
a minimal generator and both are minimal generators for $J.$ They
also studied a special case of these ideals namely, $F =
X^{[a]}H(Y,Z) + WZ^{[j-1]}$. In that case, they described completely
the minimal free resolution of $R/I$. They concluded their article
with an interesting question on classification of ideals $I$ with
$n+2, ~ n+3$ and $n+4$ number of generators where $n$ denotes the
minimum number of generators of $J$.

In this article, we consider the problem in higher embedding
dimensions and partially answer this question.   

Let $R=k[w=x_0, x_1, \dots x_n] $ be a standard graded ring of
dimension $n+1$ and $I$ a homogeneous Gorenstein ideal of height $n+1$
with $H ( R/I) = ( 1, n+1, 1+{n+1\choose 2}, \ldots , {n+1\choose
2}+1, n+1, 1) $. Suppose the ideal $RI_2$ has height one.   Then after
a possible change of variables, $I$ equals $\ann F$, where $F = G(X_1, \ldots,
X_n) + W^{[j]}$ or $F = G(X_1, \ldots X_n) +WX_n^{[j-1]}$ under
Macaulay equivalence. 

Let  $J = I \cap k[x_1, \ldots x_n]= I\cap R'$ so that $R = R'[w]$.
If $F = G(X_1, \ldots X_n) + W^{[j]}$, then $J$ is Gorenstein of
height $n$ and  $I$ is minimally generated by $J, wx_i, 1 \leq i \leq n-1,
w^{\beta}-g$
for some $g\in R'$ and not in $I$.   If $F =  G(X_1, \ldots X_n) +
WX_n^{[j-1]}$,  then $I$ is   generated by $J,wx_i, 1\le i\le n-1, w^2
, wx_n^{\beta} -g$ for some $\beta \le j-1$ and $g \in R' \backslash J$.
In the generic case, all of these are minimal generators, that is
$gx_i$ is not a minimal generator of $J$ for any $i$.  Thus, if $G$ is
sufficiently general, $\mu (I) = \mu (J) + n+1$ and we say $G$ or $I=
\ann (G+ WZ^{[j-1]})$ is {\it  generic}.  Since there is a one-to-one
correspondence between the height four Gorenstein ideals $I$ with the
property that $I_2 \cong (wx_i, 1\le i \le n-1,w^2)$ and homogeneous
forms $F = G(X_1, \ldots X_n) + WZ^{[j-1]} \in k_{DP}[W,X_1, \ldots
X_{n-1},X_n= Z]$, classifying such ideals is equivalent to classifying
these homogeneous forms in the divided power algebra. Therefore, we
try to classify the property of $I$ being generic in terms of certain
properties of the homogeneous form $F$ such that $I = \ann F$.  

In Section $1$, we begin by comparing $I = \ann(G + WZ^{[j]})$ and
$I_{x_i} = \ann(\frac{\partial{G}}{\partial X_i} + WZ^{[j-1]})$ in
general and show that when  $X_i$ divides $G$, $I$ can be generic
only if  $I_{x_i}$ is generic.  When $n= 3$, for Gorenstein Artinian
algebras, we have more specific results.  We prove that if $I =
\ann(F)$ is generic, then there are some relations among the $X, ~Y$
and $Z$-degrees of $F$.

In Section $2$, we study the case when $F = G(Y,Z) + X^{[a]}G_1 +
WZ^{[j-1]}$. When $a > j/2$ we prove that $I$ is generic if
and only if $\max \deg_Y (G) = \max \deg_Y (G_1)$. In the next
section, we discuss the general case, where $F = G(Y,Z) +
X^{[a_1]}G_1(Y,Z) + \cdots + X^{[a_n]}G_n(Y,Z) + WZ^{[j-1]}$. Assuming
that $j-a_n < a_1 < \cdots < a_n$, we obtain necessary and sufficient
conditions for $I$ to  be generic.     

We conclude our article by comparing the Hilbert functions of $R/I$
and $R'/J$. As a consequence, we show that certain classes of
Cohen-Macaulay height three ideals in $k[x,y,z]$ are unimodal, even
though they are not Gorenstein.

\section{Genericity}

Let $R$ be a standard graded $k$ algebra of dimension $n+1$ and $I$ be
a graded Gorenstein ideal of height $n+1$ such that $H(R/I) = ( 1,
n+1,1+ {n+1\choose2}, \ldots , n+1, 1)$.  Then the ideal $I$ has
$n$ quadrics amongst its minimal generators.  Suppose these quadrics
generate an ideal $I_2$ of height one, then there exists a one form
$w$ such that $I_2\subset (w)$ and $R = k[w,x_1, \ldots x_n]$ for some
suitable one form $x_i$. Let $R'= k[x_1, \ldots, x_n]$ and $J = I\cap
R'$.  

We will make use of the Macaulay equivalence between graded
Gorenstein ideals of socle degree $j$ in $R=k[w,x_1, \ldots x_n]$ and
$j$-forms in the divided power algebra $R^*=k_{DP}[X_1, \ldots, X_n]$
by the action of $R$ on $R^*$ by differentiation. If $F\in R^*$, then
$\ann F= \{ f\in R ~| ~ (\partial /\partial f)(F) = 0\}$.  

The multiplication in the divided power algebra
is different from the usual polynomial algebra : 
$X^{[a]}\cdot X^{[b]} = \frac{(a+b)!}{a!b!} X^{[a+b]}$.

In our situation, $I = \ann F$, where $F = G+W^{[j]}$ or 
$G+WX_n^{[j-1]}$ depending on whether $w^2\notin I$ or $w^2 \in I$.  
When $G$ is generic among $R[X_1, \ldots X_n]$, $\mu (I ) = \mu(J)+n+1$.
We say that $F$ is {\it generic mod W} and the corresponding ideal  
$ I = \ann F$ is {\it generic} if $\mu(I) = \mu(J) + n+1$.

If $w^2\notin I$, then we can see that $I_2 = (wx_i, 1\le i \le n)$
and $R'/J$ is a Gorenstein Artin algebra of embedding dimension $n$.
The minimal number of generators for $I$, $\mu (I)$  is precisely
$n+1$ more than that of $J$ for $I = (J, wx_1, \ldots , wx_n, w^j- g)$
for some $g\in R', g\notin J$.   Thus $\mu(I) = \mu(J)+n+1$ and in
this case, $I$ is always {\it generic} in our sense. 

From now on, we consider the case where $w^2\in I$.  Then $w^2$ is one
of the quadrics minimally generating $I$ and for a suitable choice of
one forms $x_1, \ldots x_{n-1}, x_n=z$,   $I_2\cong  (w^2, wx_1,
\ldots wx_{ n-1 })$.  
There exists a unique $F = G(X_1, \ldots, X_{n-1}, Z) + WZ^{[j-1]}$ of
degree $j$ in the divided power algebra such that $I = \ann F$. It has
been proved in \cite{E-S} that $I$ is generated by $(I_2, J,
wz^{\beta}-g)$ with $J = I \cap k[x_1, \ldots, x_n=z]$, $g \in
k[x_1,\ldots, x_{n-1},z]~ \backslash~ J$ and $\beta \leq j-1$. Suppose
$J$ is minimally generated by $\alpha_1, \ldots, \alpha_n = z^j$.
Since $wx_i$ are in $I$, $gx_i, 1\le i\le n-1$ belong to $J$.
Therefore, if $gx_i  \in \mathfrak{n} J$, where $\mathfrak{n} =
(x_1,\ldots x_{n-1}, z) \subset k[x_1, \ldots, x_{n-1}, z ]$, then
they are not part of any minimal generating set for $J$ and hence
$\mu(I) = \mu(J) + n+1$.  Depending on whether $gx_i$'s   are part of
a minimal generating set for $J$, the number of minimal generators for
$I$ will be $\mu(J) + t$ for $2\le t\le n+1$. We summarize this as:
\begin{remark}
If $A=R/I$ is a Gorenstein Artin algebra with Hilbert function $(1,
n+1, 1+{n+1\choose 2}, \ldots , n+1, 1)$, then after a linear change
of variables, $R = k[x_1, \ldots, x_n, w]$ and there exists a unique
minimal generator for I of the form $wx_n^t - g(x_1, \ldots x_n)$.  If
$J = I\cap k[x_1, \ldots , x_n]$, then $\mu(I) = \mu(J)+n+1-r$, where
$r$ is the cardinality of $\{i~|~ gx_i \in J/\m J\}$.
\end{remark}

\vskip 2mm
\noindent
Thus, $I$ is generic if and only if $gx_i \in (x_1, \ldots x_n)J$ for
all $1\le i\le n-1$.  The purpose of this paper is to classify these
polynomials $F$ that give rise to generic ideals $I$.
 
We begin by proving a result which helps us restrict our study to a
simpler class of polynomials.  Let $F = G(X_1, \ldots, X_{n-1},
Z)+WZ^{[j-1]}$ and $I = \ann F$. Without loss of generality, we may
also assume that $G$ does not have a term in $LZ^{j-1}$ where $L$ is a
linear form for in that case we can replace $W$ by $(W+L)$.   In what
follows we will take $I = \ann F$ with $F$ as above.   For any form $F
\in k_{DP}[X_1, \dots, X_{n-1}, Z, W]$, $F_X$ denotes the partial with
respect to $X$. 

\begin{theorem}\label{restrict}
Suppose $F' = X_tG+ WZ^{[j]}$  for some $t$, $1\le t\le n-1$.   If
$\ann F$ is generic, then so is $\ann F'$.
\end{theorem}

\begin{proof} 
Let $X_t   = X, ~I' = \ann F'$ and $I = \ann F$.  Let $R' =
k[x_1, \ldots, x_{n-1}, z]$ and $J' = I' \cap k[x_1, \ldots, x_{n-1}, z]$.
Then $I' = (wx_i,  1\le i\le n-1, w^2, J', wz^{\beta'} - g')$ for some
$g' \in R'~ \backslash ~J'$ and $\beta' \leq j$. We first show that
$wz^{\beta'} - g'$ can be replaced with $wz^{\beta+1} - h$ for some $h
\in R' ~ \backslash ~ J'$ and $hx \in \n J'$.

\vskip 2mm
\noindent
\textsc{Claim:} $wz^{\beta} - g \in I$ for some $g$ if and only if 
$wz^{\beta+1} - gx \in I'$.

\vskip 2mm
\noindent
\textit{Proof of the claim:} If $wz^\beta - g \in I$, then
\begin{eqnarray*}
  0 & = & \frac{\partial F}{\partial wz^\beta - g} = \frac{\pa G}{\pa
  wz^\beta} + \frac{\pa WZ^{[j-1]}}{\pa wz^\beta}- \frac{\pa G}{\pa g} 
  - \frac{\pa WZ^{[j-1]}}{\pa g} \\
  & = & Z^{[j-\beta-1]} - \frac{\pa G}{\pa g} - W\frac{\pa
  Z^{[j-1]}}{\pa g}.
\end{eqnarray*}
This implies
$\displaystyle{\frac{\pa Z^{[j-1]}}{\pa G}} = 0$ and
$\displaystyle{\frac{\pa G}{\pa g} = Z^{[j-\beta-1]}}$. Now,
consider
\begin{eqnarray*}
  \frac{\pa(G' + WZ^{[j]})}{\pa wz^{\beta+1}-gx} & = & \frac{\pa
  G'}{\pa wz^{\beta+1}} + \frac{\pa WZ^{[j]}}{\pa wz^{\beta+1}} -
  \frac{\pa G'}{\pa gx} - \frac{\pa WZ^{[j]}}{\pa gx} \\
  & = & Z^{[j-\beta-1]} - \frac{\pa G}{\pa g} \\
  & = & 0.
\end{eqnarray*}
Therefore $wz^{\beta+1} - gx \in \ann(G' + WZ^{[j]}) = I'$.

\vskip 2mm
\noindent
Suppose $wz^t - h' \in I'$ for some $t < \beta + 1$ and $h' \notin J'$.
Then 
\begin{eqnarray*}
  0 & = & \frac{\pa(G' + WZ^{[j]})}{wz^t - h'} = \frac{\pa G'}{\pa
  wz^t} + \frac{\pa WZ^{[j]}}{\pa wz^t} - \frac{\pa G'}{\pa h'} - 
  \frac{\pa WZ^{[j]}}{\pa h'} \\
  & = & Z^{[j-t]} - \frac{\pa G'}{\pa h'} - W\frac{\pa Z^{[j]}}{\pa
  h'}.
\end{eqnarray*}
Hence $\displaystyle{\frac{\pa Z^{[j]}}{\pa h'}} = 0$ and
$\displaystyle{\frac{\pa G'}{\pa h'} = Z^{[j-t]}}$. This implies
that $x$ divides $h'$. Let $h' = xh$ for some $h$. 
Consider  $wz^{t-1} - h$,
\begin{eqnarray*}
  \frac{\pa G + WZ^{[j-1]}}{\pa wz^{t-1} - h} & = & \frac{\pa G}{\pa
  wz^{t-1}} + \frac{\pa WZ^{[j-1]}}{\pa wz^{t-1}} - \frac{\pa G}{\pa
  h} - \frac{\pa WZ^{[j-1]}}{\pa h} \\
  & = &  Z^{[j-t]} - \frac{\pa G}{\pa h} - W\frac{\pa Z^{[j-1]}}{\pa
  h}.
\end{eqnarray*}
If $\displaystyle{\frac{\pa Z^{[j-1]}}{\pa h}}$
is non-zero, then $h$ contains a pure power of $z,$ which must be $z^t$. 
But $\displaystyle{\frac{\pa G'}{\pa h'} = Z^{[j-t]}}$ would then mean that $G'$ has a term $XZ^j$ which is not possible by our assumption on $F$.  

Therefore $\displaystyle{\frac{\pa Z^{[j-1]}}{\pa h} = 0}$. 
Since $G = \displaystyle{\frac{\pa G'}{\pa x}}$ and $h' = xh$, 
we get $\displaystyle{\frac{\pa G+WZ^{[j-1]}}{\pa wz^{t-1}-h} = 0}$.
Therefore $wz^{t-1} - h \in I$. This
contradicts the minimality of $\beta$. This completes the proof of the
claim.

\vskip 2mm
\noindent
Suppose $I = (I_2, J, wz^\beta - g)$ is generic
and let $I' = (I_2, J', wz^{\beta+1} - gx)$. First note that
if $h \in \n J$, then $hx_t \in \n J'$ for all $t$.   Since $gx_t \in
\n J$, $gx_sx_t \in \n J'$ for all $1\le s,t \le n$.  Therefore they are not
part of minimal generating set for $J'$. Hence $I'$ is generic.
\end{proof}

We may also restate the theorem as a necessary condition as follows:\\
 Let $I = \ann F =G+WZ^{[j]}$.  Suppose $G = X_tG_{X_t}$ for some $1\le t\le n-1$. 
 Then $I = \ann F$ is generic if $ \ann G_{X_t}+WZ^{[j-1]}$ is generic.  

As a result of the above theorem, we concentrate on polynomials $F = G
+ WZ^{[j-1]}$ in the divided power algebra such that none of the $X_i$, $1\le i \le n-1$ divides $G$ and of course that $G$ has no term containing $Z^{[j-1]}$.

The converse of the above theorem is not true in general as we can see in the Example \ref{converse}.


We remark that it is not always the case that we can achieve
'genericity' by multiplying by an $X_i$ even if one of $gx_j$ is a
non-minimal generator for the corresponding ideal $J$.  See the
examples in the last section.  However, it is an interesting question,
whether if $F = G+WX_n^{[j-1]}\in k_{DP}[X_1, \ldots X_{n}, W]$ is not
generic, does there exist a suitable power of $X_i, 1\le i\le n-1$
multiplying G by which will result in a generic ideal, better yet, 
does there exist a  suitable one
form $L(X_1, \ldots X_{n-1})$  so that $I'= \ann ( LG+WX_n^{[j]})$ is
such that $\mu (I') = n+1+ \mu ( I')\cap k[x_1, \ldots, x_n]$?


\vskip .2truein
\section {Embedding Dimension Four}
In embedding dimension 4,  we can get a stronger characterization.  
We now let $n=3$ so that $R=k[w,x,y,z]$.
\begin{notation}\label{notation-f}
For the rest of the paper, we set 
\begin{eqnarray*}
F & = &  G_0 + X^{[a_1]}G_1 + \cdots + X^{[a_l]}G_l 
+ WZ^{[j-1]} \\
 & = & \sum_{t=0}^m c_{p-t}Y^{[p-t]}Z^{[q+t]} + 
X^{[a_1]}\left(\sum_{r_{1k}+s_{1k}=j-a_1}
c_{r_{1k}}Y^{[r_{1k}]}Z^{[s_{1k}]}\right) \\ & & +
  \cdots +  X^{[a_n]}\left(\sum_{r_{nk}+s_{nk} = j-a_n}
  c_{r_{nk}}Y^{[r_{nk}]}Z^{[s_{nk}]}\right) +WZ^{[j-1]}, 
\end{eqnarray*}
where $c_{p} \neq 0$, $a_1 <  \cdots < a_n$ and one of the $G_i's$
contain a pure power of $Z$.\\
\end{notation}

For a polynomial $h(x_1, \ldots,
x_l) \in k[x_1, \ldots, x_l]$, let $\deg_{x_i} h$ denote the highest
power of $x_i$ in $h$.   
Thus, in $F$ as in \ref{notation-f}, $\deg _X F = a_n,$.

\vskip .2truein
We first obtain a necessary condition for an ideal to be generic
in terms of $F$. Before we prove the result, 
we prove a technical, but very important lemma that is needed
in the proof of this theorem. This lemma will play a crucial role in the
proofs of all the forthcoming characterization results as well.
 
\begin{lemma} \label{lem0} With the notation as in
\ref{notation-f}, if $\deg_Y G_0 = p$
then $\ann G_0$ doesn't have elements of degree less than or equal to
$j-p$ other than $y^{p+1}$ when $p < j-p$.
\end{lemma}

\begin{proof} 
Clearly $y^{p+1} \in \ann G_0$.
Suppose $\ann G_0$ contains a generator of degree less than $q$, say
$g = \sum_{i=0}^k \alpha_i y^{r-i}z^i$ for some $0 \leq k \leq r$,
where $\alpha_k \neq 0$.
Then we have
\begin{eqnarray*}
  0 = \frac{\partial G_0}{\partial g} 
  & = & \frac{\pa c_pY^{[p]}Z^{[q]} + c_{p-1}Y^{[p-1]}Z^{[q+1]}+
  \ldots + c_{p-m}Y^{[p-m]}Z^{[q+m]} }{\pa  \alpha_r y^{r}+
  \alpha_{r-1} y^{r-1}z+ \cdots +\alpha_{k} y^{r-k}z^{k}} \\
  & = & \alpha_kc_p Y^{[p-r+k]}Z^{[q-k]}+\mbox{terms in Y of degree less
  than } p-r+k.
\end{eqnarray*}
Since $c_p \neq 0$, $\alpha_k=0$, which is a contradiction. Hence the
assertion follows.
\end{proof}

\begin{theorem} \label {th0}
With the notation as in \ref{notation-f}, if $I$ is generic,
then either $a_1 \leq \deg_Y G_0$ or
$\deg_Y G_0 \leq \max \{\deg_Y G_i ~ : ~ i = 1, \ldots, n\}$.
\end{theorem}

\begin{proof}
We first show that if $f \in \ann
F$ with $\deg f < a_1$, then either $f \in \ann G_0 \cap \ann WZ^{[j-1]} \cap \ann (\sum_{i=1}^n
X^{[a_i]}G_i)$  or $f \in \ann(G_0+WZ^{[j-1]}) \cap \ann (\sum_{i=1}^n
X^{[a_i]}G_i).$ For if $f \in \ann F$, then 
$$ 
0 = \frac{\partial F}{\partial f}  =  \frac{\partial G_0}{\partial f} +
\sum_{i=1}^n \frac{\partial X^{[a_i]}G_i}{\partial f} + \frac{\partial
WZ^{[j-1]}}{\partial f}.
$$
Therefore
$$
 \frac{\partial G_0}{\partial f} +  \frac{\partial WZ^{[j-1]}}{\partial f}= 
- \sum_{i=1}^n\frac{\partial X^{[a_i]}G_i}{\partial f}.
$$
Since $G_0$ is a polynomial in $Y$ and $Z$, the term on the left hand
side of the above equality does not involve $X$. Since $\deg f <
a_1$ we get
that $\sum_{i=1}^n \frac{\partial X^{[a_i]}G_i}{\partial
f}  = 0 $ and $\frac{\partial (G_0+WZ^{[j-1]})}{\partial f}=0$.
We now proceed to the proof of the theorem. We show that if $a_i >
p > r_{ik}$ for every $i$ and $k$, where $p = \deg_Y G_0$, then $I$
is not generic.  We do this by considering the following two cases:

\vskip 2mm
\noindent
{\sc Case 1.} Suppose $p \leq j-p$. By Lemma \ref{lem0}, $y^{p+1}$ is a
minimal generator for $\ann G_0$ since all other generators have degree
bigger than $j-p$.  Since $p>r_{ik}$ for all $i, k$, it can be seen that
$y^{p+1} \in \ann G_0 \cap \ann WZ^{[j-1]} \cap \ann X^{[a_i]}G_i$ for
all $i$. We then conclude that $y^{p+1}$ is a minimal generator for
$\ann F$.\\
On the other hand, note that $f= wz^{p-1} - y^{p} \in \ann F$.  Since
the degree of $f=p<a_1$ and contains a term in $w$, we
get $f \in
\ann(G_0+WZ^{[j-1]})\cap \ann (\sum_{i=1}^l X^{[a_i]}G_i)$. We show
that $f$ is a minimal generator. Suppose $wz^\beta - g \in I$ for some
$\beta < p-1$. It implies that $y^p - z^{p-1-\beta}g \in I$ which is
impossible since $y^{p+1}$ is a minimal generator.  Therefore,
$wz^\beta - g \notin I$ for any $\beta < p-1$  so
that $f$ is a minimal generator for $I$.  Since $y^{p+1}$ is a
minimal generator,  $I$ is not generic.

\vskip 2mm
\noindent
{\sc Case 2.} Suppose $p > j-p$. In that case $wz^{p-1}-y^p \in \ann F$, but might not be a minimal generator for $I$. Let $f =wz^{\beta-1} - g(x,y,z)$  be a
minimal generator for $I$, with $\beta  \leq p$. We first show that we can 
replace $g(x,y,z)$ by $g(y,z)$. We write $g(x,y,z)=g_1(y,z)+xg_2(x,y,z)$. Then, 
$$\frac{\partial F}{\partial f} =-\frac{\partial G_0}{\partial
g_1(y,z)} - \displaystyle \sum_{i=1}^n \frac{\partial
  X^{[a_i]}G_i}{\partial g_1(y,z)+xg_2(x,y,z)} + \frac{\partial
	WZ^{[j-1]}}{\partial wz^{\beta-1}}=0.$$
Since $a_i >p \geq  \beta$, we get $\frac{\partial G_0}{\partial g_1}=
Z^{[j-\beta]}$ and $\displaystyle \sum_{i=1}^n \frac{\partial
  X^{[a_i]}G_i}{\partial g_1+xg_2} =0$. Therefore, either $g_1=y^p$ or
  the degree of $g_1$ is strictly less than $p$. But in both cases and
  by Lemma \ref{lem0}, the degree of $g_1$ is greater than $j-p$. On
  the other hand, the degree of $G_i$ is at most $j-p-1$. It follows
  that  $\frac{\partial X^{[a_i]}G_i}{\partial g_1}=0$ for all $i$,
  and hence $ xg_2 \in J$. So $g(x,y,z)$  can be replaced by $g(y,z)$.
  \\
We know
that $\frac{\partial F}{\partial g(y,z)} = Z^{[j-\beta]}$ and
$yg(y,z)\in \ann F$. In fact, $yg(y,z)\in \ann G_0 \cap \ann
(WZ^{[j-1]}) \cap \ann (\sum_{i=1}^n X^{[a_i]}G_i)$. As in the case of the proof of Lemma $4.2$ in \cite{E-S}, it can be seen that
 $\ann G_0$ is minimally generated in $k[y,z]$ by a regular
sequence $(y\delta(y,z), \theta(y,z))$ with $\theta(y,z)= z^t +
\theta_1(y,z)$.  By Lemma \ref{lem0}, 
the degrees of $y\delta(y,z)$ and $\theta(y,z)$ are at least
$j-p+1$ and the degrees of the $G_i's$ are at
most $j-p-1$. It follows that $y\delta(y,z)$ and $\theta(y,z) \in  \ann G_0 \cap \ann  (\sum_{i=1}^n X^{[a_i]}G_i)$. But $\theta(y,z) \notin \ann WZ^{j-1}$. Therefore 
$\theta(y,z) \notin \ann F$, but $y\theta(y,z) \in \ann F$.
Hence, $y\delta(y,z)$ and $y\theta(y,z)$ are minimal generators for $\ann F$.\\
We now show that $yg(y,z)$ is a minimal generator for $I$ and can be
chosen to be $y\delta(y,z)$. Suppose $yg(y,z)$ is not minimal then
$yg(y,z)= f_1(y,z)y\delta(y,z)+f_2(y,z)y\theta(y,z)$. It implies that
$g(y,z)= f_1(y,z)\delta(y,z)+f_2(y,z)\theta(y,z)$. Consider
$$\frac{\pa G_0}{\pa g(y,z)}= \frac{\pa G_0}{\pa
f_1(y,z)\delta(y,z)+f_2(y,z)\theta(y,z)}=  \frac{\pa G_0}{\pa
f_1(y,z)\delta(y,z)}=Z^{[j- \beta]}.$$
We can choose $g(y,z)= f_1(y,z)\delta(y,z)$. If $f_1(y,z)$ is a constant then we are done. Otherwise $f_1(y,z) =
cz^u$ since $y\delta(y,z)$ is in $J$. On the other hand,
$\delta(y,z)$ cannot have a pure power of $z$, otherwise $g(y,z)$ will
have a pure power of $z$ and will not belong to $I$. But $\frac{\partial G_0}{\partial \delta(y,z)}= Y\frac{\partial G_0}{\partial y \delta(y,z)} +\alpha Z^{[\gamma]}= \alpha Z^{[\gamma]}$ with $\alpha \neq 0$ and the multiplication is the usual polynomial multiplication.
 It implies that  $f'= wz^{\beta-1-u} -
\delta(y,z)$ is a minimal generator for $I$ which contradicts the
minimality of $f$. Hence, $g(y,z)= c\delta(y,z)$ and $yg(y,z)$ is
minimal. Hence $I$ is not generic.
\end{proof}

%
In \cite{E-S}, it was proved that if $F = X^{[a_1]}G_1 + WZ^{[j-1]}$
with $a_1 \geq 1$, then $I$ is generic.
Example \ref{example2.3} shows that this property does not hold if $F
= G_0 + X^{[a]}G_1 + WZ^{[j-1]}$. 
 Let 
\begin{eqnarray*}
F & = & G_0 + X^{[a]}G_1+WZ^{[j-1]} \\ 
 & = & \sum_{r=0}^m c_{p-r}Y^{[p-r]}Z^{[q+r]} + 
X^{[a]}\left(\sum_{r_{i}+s_{i}=j-a}
\alpha_{r_{i}}Y^{[r_{i}]}Z^{[s_{i}]}\right) + WZ^{[j-1]}.
\end{eqnarray*}

If we suppose that $a>j-a \geq p$, then we get an improved version of
Theorem  \ref {th0} in this case.
 
\begin{theorem} \label {thm1} 
Let $F$ be as above and suppose $\deg_X F > \deg G_1 \geq \deg_Y G_0$. Then
$I$ is generic if and only if $\deg_Y G_1 = p$.
\end{theorem} 

\begin{proof}
  Suppose $\deg_Y G_1 = p$.  We write $F  =  c_p^{-1}Y^{[p]}Z^{[a+s]}+\cdots+
c_2Y^{[2]}Z^{[a+s+p-2]} +X^{[a]}(\alpha_pY^{[p]}Z^{[s]}+
 \cdots+ \alpha_0 Z^{[p+s]} )+WZ^{[j-1]}$
with $c_p, \alpha_p$ and $\alpha_0 \neq 0$.\\
We first show that $wz^{p+s} -c_p^{-1}y^{p}z^{s+1}$ is a minimal generator
for $I$. Clearly, $wz^{p+s} -c_p^{-1}y^{p}z^{s+1} \in I.$
Suppose there exists $wz^{\beta}-g \in I$ for some $\beta < p+s$. As
in the proof of Theorem \ref{th0}, we can assume that $g$ is a
function of  $y$ and $z$. 
Then we have
\begin{eqnarray*}
  0 & = & \frac{\partial F}{\partial wz^{\beta} -g} \\
& = & \frac{\partial G_0+X^{[a]}G_1+WZ^{[j-1]}}{\partial wz^{\beta} -g}. 
\end{eqnarray*}
Therefore
$$\frac{\partial G_0+X^{[a]}G_1}{\partial g} =Z^{[j-1-{\beta}]}=
Z^{[a]}.$$

Since $G_0$ does not involve $X$ and $g$ has degree $p+s<a,$ we get
$\frac{\partial G_0}{\pa g}= Z^{[a]}$ and $\frac{ \pa G_1}{\partial g}
=0$. Therefore, there exists $0 \leq i \leq p-2<a$ with $c_{p-i} \neq 0$
such that $g = c_{p-i}^{-1}y^{p-i}z^{s+i}+ g_1(y,z)$.  If $i \neq 0,$ then
$\frac{\partial G_0}{\partial g}=  c_p.c_{p-i}^{-1} Y^{[i]}Z^{[a-i]}+ \cdots +Z^{[a]}+
\cdots $ and $g_1$ does not annihilate $G_0-c_{p-i}Y^{[p-i]}Z^{[a+s+i]}$
by Lemma \ref{lem0}, which is a contradiction. Hence $i=0$ and $g=
c_{p}^{-1}y^{p}z^{s}+ g_1(y,z)$. Therefore, $\frac{\partial G_0}{\partial g}=
Z^{[a]}$ and $\frac{\partial X^{[a]}G_1}{\partial g}= c_p^{-1} \alpha_p
X^{[a]}$. 
This implies that there exists a term $g_2(y,z)$ in $g_1(y,z)$ in $y$ and
$z$ of total degree $p+s$ such that $\frac{\pa X^{[a]}G_1}{\pa
g_2}=-c_p^{-1}\alpha_pX^{[a]}$. 
Again $\frac{\partial G_0-c_pY^{[p]}Z^{[a+s]}}{\partial g_2} \neq 0$ by
Lemma \ref{lem0}. Continuing in this manner, we see that
$wz^\beta - g$ does not annihilate $F$, which is a contradiction. Therefore $p+s$ is the
minimum exponent $\beta$ of $z$ such that $wz^\beta -g \in I$ and
hence $wz^{p+s} -c_p^{-1}y^{p}z^{s+1}$ is a minimal
generator in $I$.
\vskip 2mm
\noindent
Since $y^{p+1} \in I$, $y\cdot y^p z^{s+1}$ is not a minimal generator
of $I$. Note also that if $j$ is the smallest integer so that
$\alpha_{p-j} \neq 0$, then $x(\alpha_{p-j}y^{p}z^{s}-\alpha_p
y^{p-j}z^{s+j}) \in I$. Therefore, $x\cdot y^pz^{s+1}$ is also not a
minimal generator of $I$. Hence $I$ is generic.

\vskip 2mm
\noindent
We now prove the converse. We know by Theorem \ref{th0} that if $a>p$
and $I$ is generic, then $\alpha_{p+i} \neq 0$
for some $i \geq 0$. With the extra assumption that $F$ has only two
terms and $a>j-a$, we show that $\alpha_p \neq 0$ and
$\alpha_{p+i} =0$ for $i>0$. So it suffices to show that if there exists
an $i>0$ such that $\alpha_{p+i} \neq 0$, then $I$ is not generic. 
Suppose $\alpha_{p+i} \neq 0$ for some $i$. Let 
\begin{eqnarray*}
F  & = & c_pY^{[p]}Z^{[a+s]}+\cdots +
c_2Y^{[2]}Z^{[a+s+p-2]} 
+X^{[a]}\left(\sum_{m=0}^{p+r} \alpha_{m}Y^{[m]}Z^{[p+s-m]}\right)
+WZ^{[j-1]},
\end{eqnarray*}
where $\alpha_{p+i} \neq 0$ for some $i>0$. It is clear that
$wz^{p+s}-{c_p}^{-1}y^pz^{s+1} \in I$. If it is not a minimal generator
for $I$, then there exists $\beta < p+s$ such that $wz^{\beta}-g(y,z)$
is minimal. Then we have 
\begin{equation*}  
\frac {\partial F}{\partial wz^{\beta}-g(y,z)}= Z^{[j-1-\beta]}
-\frac{\partial G_0}{\partial g}-X^{[a]} \frac{ \partial G_1}{\partial
g}  =0. \label{eq1} 
\end{equation*}
Since $a > p+s$, we get $ \frac{ \partial G_1}{\partial g}  =0 $ and $ \frac{\partial
G}{\partial g}=Z^{[j-1- \beta]}$.
We first claim that $g(y,z)= {c_p}^{-1}y^{p}z^{q}+g_1(y,z)$, where $p+q=
\beta+1$ and the exponents of $y$ in $g_1(y,z)$ are greater than $p+1$.
Let $g(y,z)= c_{p-i}^{-1}y^{p-i}z^{q} +g_1(y,z)$, with
$i \geq 0$ and the exponents of $y$ in $g_1$ strictly bigger than $p-i$. 
Then, $\frac{\pa (G_0-
c_{p-i} Y^{[p-i]}Z^{[a+s+i]})}{\pa g(y,z) }=0$.  But by Lemma
\ref{lem0}, $\ann(G_0- c_{p-i} Y^{[p-i]}Z^{[a+s+i]})$ doesn't contain 
generators of degree less than $a+s$. This implies that
$i=0$. Hence $g(y,z)$ is of the required form.

Let $q$ be the minimal power of $z$ such that
$wz^{\beta}-c_p^{-1}y^pz^{q}-g_1(y,z)$ is a minimal generator for $I$ with
$g_1 \in \ann G$ and $c_p^{-1}y^pz^q+g_1  \in \ann G_1$. 
We show that $f=y(c_p^{-1}y^pz^{q}+g_1(y,z))$  is a minimal generator for
$I$. 
Suppose not, then $f$ can be obtained  from a
combination of a generator of the form $y^{p+1}z^{q-k} + g_2(y,z)$ and
other  generators of $I$.  Note that $y^{p+1}z^{q-k}+g_2(y,z) \in \ann
G_0$ and hence belongs to $\ann G_1$.  We have also seen that $c_p^{-1}y^pz^{q}+g_1(y,z)) \in \ann G_1$.  
It is clear that neither of the generators $y^{p+1}z^{q-k}+g_2(y,z)$
and $c_p^{-1}y^pz^q+g_1(y,z)$ can be obtained from each other. We then
study the minimal generators  of $\ann G_1$. We know that $ \ann G_1$
is minimally generated by a regular sequence of the form
$(y\delta(y,z), \theta(y,z))$ with $\theta(y,z)=z^t+\theta_1(y,z)$.
For these generators to be in $I$, no pure power of $z$ of degree less
than $j$ and no terms with degree of $y$ less than $p$ can appear in
any of them. In that case, $y\delta(y,z)$ and $\theta(y,z)$ will be
multiplied by the appropriate power of $y$ to get rid of the power of
$z$ and have all powers of $y$ in their terms greater or equal than
$p+1$. Since both generators $y^{p+1}z^{q-k}+g_2(y,z)$ and
$c_p^{-1}y^pz^q+g_1(y,z)$  are independent in $\ann G_1$,
$y^{p+1}z^{q-k}+g_2(y,z)$ and $y(c_p^{-1}y^pz^q+g_1(y,z))$ are
independent in $I$.  Hence $f$ is minimal and $I$ is not generic.  
\end{proof}


Example \ref{ex1section2} shows that statement of Theorem \ref{thm1} 
does not hold if $p<a<p+s $.

\vskip 3mm

Now, we will consider the general case in codimension four.  Let F be as in Notation \ref{notation-f}. 
We now obtain a  generalization of Theorem \ref{thm1}.
Along with Notation \ref{notation-f} we further assume that $G_n$ has
a pure power of $Z$ (in other words, $Y$ does not divide $G_n$).
We first prove a necessary condition for an ideal to be generic.

\begin{theorem} \label{tm2} 
Let $F$ be as in Notation \ref{notation-f} with $j-a_n < a_1$,
$Y$ not dividing $G_n$ and $\deg_Y G_i <\deg_Y G_0$ for all
$i=1,\ldots, n-1$.  If $I$ is generic, then
$\deg_Y G_n  = \deg_Y G_0$.
\end{theorem}

\begin{proof} 
By Theorem \ref{th0}, it
suffices to show that if  there exists an $k$ such that $r_{ik}>p$ in
$G_n,$ then $I$ is not generic.
%
%
It is clear that $wz^{p+s}-c_{p}^{-1}y^pz^{s+1} \in I$. Suppose it is not a minimal generator for $I$, then there exists $\beta < p+s$ such that $wz^{\beta}-g(y,z)$ is minimal.\\
 Since  
 $$ 0 = \frac {\partial F}{\partial wz^{\beta}-g(y,z)}= Z^{[j-1-\beta]} -\frac{\partial G}{\partial g}-\sum_i \frac{ \partial X^{[a_i]}YG_i}{\partial g} -X^{[a_n]}\frac{\pa G_n}{\pa g}$$ we get  
 $ \frac{\partial G}{\partial g}=Z^{[j-1- \beta]}$ and the rest is  zero.\\
As in the case of the proof of Theorem \ref{thm1}, we can see that
$g(y,z)= c_{p}^{-1}y^{p}z^{q}+g_1(y,z)$ with $p+q= \beta+1$
 and $g_1(y,z)$ having all powers of $y$ greater than $p+1$.
 We note that $g_1 \in \ann G_0$ and $c_p^{-1}y^{p}z^q+g_1  \in \ann G_n
 \cap \ann G_i$ for all $i=1, \ldots, n-1$ in that case. It can be
 seen that $f=y(c_p^{-1}y^{p}z^{q}+g_1(y,z))$ is a minimal generator for
 $I$. For, if not, then $f$ can be expressed as a linear combination
 of other generators of $I$, one of which is of the form
 $y^{p+1}z^{q-k} + g_2(y,z)$.
  Note that $y^{p+1}z^{q-k}+g_2(y,z) \in \ann G_0 \cap 
  ( \cap_i\ann G_i)$ and $c_p^{-1}y^pz^{q}+g_1(y,z)) \in \cap_i\ann G_i$.  
It is clear that neither of the generators $y^{p+1}z^{q-k}+g_2(y,z)$ and
$c_p^{-1}y^pz^q+g_1(y,z)$ can be obtained from each other. 

Again, looking at  the minimal generators  of $\ann G_n$,  we know they are a regular sequence of the following form  $(y\delta(y,z), \theta(y,z))$ with 
$\theta(y,z)=z^t+\theta_1(y,z)$. For these generators to be in $I$, they must be
 multiplied by the appropriate power of $y$ to get rid of the 
power of $z$ and have all powers of $y$ in their terms greater or 
equal than $p+1$. Since both generators $y^{p+1}z^{q-k}+g_2(y,z)$ and
$c_p^{-1}y^pz^q+g_1(y,z)$  are independent in $\ann G_n$,
$y^{p+1}z^{q-k}+g_2(y,z)$ and $y(c_p^{-1}y^pz^q+g_1(y,z))$ are independent in $I$.
Hence $f$ is minimal and $I$ is not generic.
\end{proof}


Examples \ref{ex1thm2} and \ref{ex2thm2} show that the converse of Theorem \ref{tm2} is not
true in general.
These examples suggest that we have to add more conditions on the
exponents to obtain a converse of Theorem \ref{tm2}. We prove the
converse with some extra assumptions.

\begin{theorem}  
Suppose $\deg_Y G_n = p, ~\deg_Z G_n = p+s, ~\deg G_n < a_1$
and $\deg_Y G_i < \deg_Y G_0$ for all $i=1,\ldots, n-1$. If $F$ satisfies one of the
two conditions below, then $I$ is generic.
\begin{enumerate}  
  \item  $\deg_Y G_i < \deg_Y (G_n - Y^pZ^s)$;
  \item $\deg_Z G_i < \deg_Z G_n$ for all $i <n$.
\end{enumerate}
\end{theorem}

\begin{proof}  
  We first show that $wz^{p+s} -c_p^{-1}y^{p}z^{s+1}$ is a minimal generator
for $I$. 
Suppose not, let $wz^{\beta}-g \in I,$ where $\beta < p+s$. 
Then
\begin{eqnarray*}
0 & = & \frac{\partial F}{\partial wz^{\beta} -g} \\
 & = & \frac{\partial [G_0(Y,Z)+X^{[a_1]}G_1+ \cdots +
 X^{[a_{n}]}G_{n}+WZ^{[j-1]}]}{\partial wz^{\beta} -g}. 
\end{eqnarray*}
Therefore
$$\frac{\partial [G_0(Y,Z)+X^{[a_1]}G_1+ \cdots +
X^{[a_{n}]}G_{n}]}{\partial g} =Z^{[j-1-{\beta}]}$$

Since $a_i>j-a_n=p+s$, $g$ may be assumed to be a function of $y$
and $z$. Hence we get
$$\frac{\partial G_0(Y,Z)}{\partial g(y,z)} =Z^{[j-1-{\beta}]}
\hspace{0.4cm} \mbox{and} \hspace{0.4cm} \frac{\partial
X^{[a_i]}YG_i(Y,Z)}{\partial g(y,z)} =0.$$ 

By Lemma \ref{lem0}, $\ann G_0(Y,Z)$ does not contain generators of
degree less than $a_n+s = j-p$. Therefore $g=c_p^{-1}y^{p}z^{s}+g_1$ for some
polynomial $g_1$. But in that case, we have
$\frac{\partial X^{[a_n]}G_n(Y,Z)}{\partial c_p^{-1}y^{p}z^{s}}= \alpha_pc_p^{-1} X^{[a_n]}$
and hence $g_1$ must contain a term $g_2$ with the power of $y$ less than
$p$ such that $\frac{\partial X^{[a_n]}G_n(Y,Z)}{\partial g_2}=-
X^{[a_n]}$.
Continuing in this manner, we see that $wz^\beta - g$ can not annihilate
$F$, which is a contradiction. Therefore $wz^\beta -g \not \in I$ for
$\beta < p+s$
and hence $wz^{p+s} - c_p^{-1}y^{p}z^{s+1}$ is a minimal generator.
 
Note that since $y^{p+1} \in I$, $y\cdot y^{p}z^{s+1}$ is not a minimal
generator. Also we have in case $(1)$ that
$x(\alpha_{p}y^{p-k}z^{s+k}-\alpha_{p-k}y^{p}z^{s}) \in I$ and in case
$(2)$, $x(\alpha_{p}z^{s+p}-\alpha_0y^{p}z^{s}) \in I$, which implies
that $xy^pz^s$ is not a minimal generator of $I$.  Hence $I$ is
generic.
 \end{proof}
\section{Hilbert Functions of $I$ and $J$}

In this section we compare the Hilbert functions of $I$ and $J = I
\cap k[x,y,z]$. It can be noted that these results are independent of
whether the ideal is generic or not.
\begin{proposition}\label{hilb-fn}
Let $R = k[x,y,z,w]$, $F = G + WZ^{[j-1]}, ~ I = \ann(F)$ and $J = I
\cap k[x,y,z]$. Then $H_{R/I} - H_{R'/J} = [0,1,\ldots,1,0,\ldots]$,
where the last $1$ occurs at the degree $\beta$.
\end{proposition}
\begin{proof}
Let $I = (I_2, J, wz^\beta - g)$, where $g \in k[x,y,z] ~\backslash~
J$. Clearly $H_{R/I}(0) - H_{R'/J}(0) = 0$ and $H_{R/I}(1) -
H_{R'/J}(1) = 1$. Note that $I_2 = \langle wx, wy, w^2 \rangle$ and
$J$ has generators of degree at least $3$. Therefore $H_{R/I}(2) = 10
- 3 = 7$ and $H_{R'/J}(2) = 6$. For $3 \leq n \leq \beta$, $I_n$, as a
$k$-vector space, is generated by $I_2R_{n-2}$ and $J_n$. If a
monomial $m_1 \in [R/I]_n$ is a $k$-basis element, then either
$m = wz^{n-1}$ or $w$ does not divide $m$ and $m$ is a $k$-basis
element of $[R'/J]_n$. Therefore, $\dim_k [R/I]_n = \dim_k [R'/J]_n +
1$ for all $3 \leq n \leq \beta$. For $n = \beta+1$, one can see that
$wz^\beta - g \in I$ and $g \notin I$ and hence $g \notin J$.
Moreover, for $n \geq \beta + 1$, $wz^n = z^{n-\beta-1}g (\mod~ I)$
and $z^{n-\beta-1}g \in I$ if and only if $z^{n-\beta-1}g \in J$.
Therefore, $H_{R/I}(n) - H_{R'/J}(n) = 0$ for $n \geq \beta + 1$.
Hence the assertion is proved.
\end{proof}

It is known that height three Gorenstein ideal in
$k[x,y,z]$ is unimodal, see \cite{s}.  As a corollary of the previous
result, we obtain a class of Artinian level algebras of embedding
dimension three, namely type two and at least one of the socle
elements is a pure power of a one form, having unimodal Hilbert function.

\begin{corollary}
If $J \subset R'= k[x,y,z]$ is an ideal such that $J = \ann(G, Z^{[\deg
G-1]})$ for some polynomial $G$ in the divided power algebra
$k_{DP}[X,Y,Z]$, then the Hilbert function of $J$ is unimodal.
\end{corollary}
\begin{proof}
Let $J$ be an ideal of the given form. Let $I = G + WZ^{[j-1]}$. Then
it can be seen that $J = I \cap R'$. From Proposition \ref{hilb-fn},
it follows that $H_{R'/J} = H_{R/I} - [0,1,\ldots, 1, 0..],$ where the
last $1$ occurs at the degree $\beta$. Since $I$ is a Gorenstein ideal
with initial degree $2$, $H_{R/I}$ is unimodal, \cite[Theorem
3.1]{mnz}. Therefore $H_{R'/J}$ is unimodal.
\end{proof}
\section{Examples}

In this section we provide some examples to illustrate our results.
We follow the notation that was set earlier, i.e., given a homogeneous
form $F$ in the divided power algebra $k_{DP}[W,X_1,\ldots, X_n]$, $I
= \ann(F) \subset k[w, x_1, \ldots, x_n]$ and $J = I \cap k[x_1,
\ldots, x_n]$. The following is an example of Theorem
\ref{restrict}. 

\begin{example} \label{1}
Let $F'= G + WZ^{[9]} =
X^{[8]}Y^{[2]}+X^{[3]}Y^{[3]}Z^{[4]}+X^{[2]}YZ^{[7]}+WZ^{[9]}$ and $F=
\frac{\partial^2 G}{\partial X^2} + WZ^{[7]} =
X^{[6]}Y^{[2]}+XY^{[3]}]Z^{[4]}+YZ^{[7]}+WZ^{[7]}$. If $I' = \ann(F')$
and $I = \ann(F)$, then it can be seen that $\mu(I')=13$ and
$\mu(J')=9$ and $\mu(I)=9$ and  $\mu(J)=5$.
\end{example}

The following example shows that the converse of Theorem
\ref{restrict} is not true in general. 

\begin{example}{\label{converse}}
Let $I =
\ann(Y^{[5]}Z^{[8]}+ Y^{[4]}Z^{[9]}+ X^{[5]}(Y^{[6]}Z^{[2]}+ Y^{[7]}Z+
Y^{[8]}+ Z^{[8]})+ WZ^{[12]})$
and $I' =
\ann(XY^{[5]}Z^{[8]}+ XY^{[4]}Z^{[9]}+ X^{[6]}(Y^{[6]}Z^{[2]}+
Y^{[7]}Z+Y^{[8]}+ Z^{[8]})+ WZ^{[13]})).$
Then $\mu(I) = 11, ~ \mu(J) = 8, ~ \mu(I') = 13$ and $\mu(J') = 9$.
Therefore $I'$ is generic, but $I$ is not.
\end{example}

The following is an example where $\ann(G+WZ^{[j-1]})$ is not generic
and $\ann(X^tG + WZ^{[t+j-1]})$ is also not generic for some $t$.
\begin{example}
Let $F = G + WZ^{[12]} =
Y^{[3]}Z^{[10]}+Y^{[4]}Z^{[9]}+Y^{[5]}Z^{[8]}+ Y^{[2]}Z^{[11]}+
X^{[6]}(Y^{[4]}Z^{[3]}+Y^{[2]}Z^{[5]}+Z^{[7]})+X^{[7]}(Y^{[3]}Z^{[3]}+
YZ^{[5]}+Z^{[6]})+WZ^{[12]}$.
It can be seen that $I = \ann(F)$ is not generic. Moreover,
$\ann(X^tG + WZ^{[12+t]})$ is not generic for $t = 1,2,3$ and that $I =
\ann(X^4G + WZ^{[16]})$ is generic. In a similar manner, one can see
that $\ann(Y^tG + WZ^{[12+t]})$ is also not generic for $t = 1,2,3,4$
and that $\ann(Y^5G + WZ^{[17]})$ is generic.
\end{example}

The following example, in codimension five, shows that $I = \ann(G
+ WZ^{[j-1]}) \subset k[w,t,x,y,z]$ need not be generic but $I =
\ann(TG + WZ^{[j]})$ is generic.

\begin{example}
Let $I =
\ann(Y^{[3]}Z^{[9]}+Y^{[4]}Z^{[8]}+T^{[3]}Z^{[9]}+X^{[3]}(Y^{[4]}Z^{[4]}T+Z^{[9]}+
+WZ^{[11]})$ and $I' =
\ann(XY^{[3]}Z^{[9]}+XY^{[4]}Z^{[8}+XT^{[3]}Z^{[9]}+X^{[4]}(Y^{[4]}Z^{[4]}T+Z^{[9]}+WZ^{[12]})).$
Then $\mu(I) = 14, ~ \mu(J) = 10, ~ \mu(I') = 16$ and $\mu(J') = 11$.
Here, $I$ is not generic as $t^4$ is a minimal generator of $J$.  But
$I'$ is generic for $t^4$ is no longer the obstruction, for
$wz^3-xt^3$ is the minimal generator of $I'$. 
\end{example}

The following is an examples of an ideal which
is not generic when $a>p$ and $p>r_{ik}$.

\begin{example} 
If $F = Y^{[10]}Z^{[5]}+ Y^{[9]}Z^{[6]}+
Y^{[6]}Z^{[9]}+X^{[11]}(Y^{[2]}Z^{[2]}+Z^{[4]})+WZ^{[14]}$, then
$\mu(I)=9$ and $\mu(J)=6$. Here
$wz^{7}-y^8+y^7z-y^6z^2+y^5z^3-yz^7$ is a minimal generator for $I$.
It can also be seen that $\ann G$ is minimally generated in $k[y,z]$ by
$(\theta, \delta_1)=(y^8+y^7z-y^5z^3-2y^4z^4-y^3z^5+y^2z^6+3yz^7+3z^8,
y^7z^2-y^4z^5-y^3z^6+yz^8+2z^9)$. Note that we may 
replace $\delta_1$ by
$3\delta_1+(2z+y)\theta = yg(y,z)$ which makes $yg(y,z)$ a minimal
generator (see the proof of Theorem \ref{th0}).
\end{example}

We have noticed that if $p \leq r_{ik}$ for some $i$ and $k$, then the 
classification is more complicated. It does not only depend on the powers
alone, but also on the coefficients as we see in the example below:

\begin{example} If $I = \ann(Y^{[3]}Z^{[8]}+Y^{[4]}Z^{[7]}+X^{[5]}(Y^{[4]}Z^{[2]}+Y^{[3]}Z^{[3]}+Z^{[6]})+WZ^{[10]} )$ then $\mu (I)= 9$ and $\mu(J)=5$ whereas if  $I = \ann(Y^{[3]}Z^{[8]}+Y^{[4]}Z^{[7]}+X^{[5]}(2.Y^{[4]}Z^{[2]}+Y^{[3]}Z^{[3]}+Z^{[6]})+WZ^{[10]} )$ then $\mu(I)= 11$ and $\mu(J)=8$.
\end{example}

%

Examples \ref{exsection2} and \ref{example2.3} are examples of  Theorem \ref{thm1}.
\begin{example}\label{exsection2}
Let $F=Y^{[4]}Z^{[10]}+Y^{[3]}Z^{[11]}+X^{[8]}
(Y^{[4]}Z^{[2]}+Y^{[3]}Z^{[3]}+Y^{[2]}Z^{[4]}+Z^{[6]})+WZ^{[13]}$.
In this case, $deg_YG_1 =p$ and one can see that the ideal $I = \ann(F)$ is
generic with $\mu(I)=11$ and $\mu(J)=7$.
\end{example}

\begin{example}\label{example2.3}
Let $F=Y^{[4]}Z^{[10]}+Y^{[3]}Z^{[11]}+X^{[8]}
(Y^{[5]}Z+Y^{[4]}Z^{[2]}+Y^{[3]}Z^{[3]}+Y^{[2]}Z^{[4]}+Z^{[6]})+
WZ^{[13]}$. In this case $deg_Y G_1 >p$ and we have $\mu(I)=11$ and $\mu(J)=8$.
\end{example}
 
Example below shows that Theorem \ref{thm1} doesn't work if
$p\leq a$: 

\begin{example}\label{ex1section2}
Let $ F = Y^{[4]}Z^{[7]}+Y^{[3]}Z^{[8]}+X^{[4]}(Y^{[4]}Z^{[3]}+
Z^{[7]})+ WZ^{[10]}$. Here, $a < \deg_X F = 4 < \deg G_1 = 7$. One can see
that $\mu(I) = 9$ and $\mu(J)=6$. Whereas if $F = Y^{[4]}Z^{[7]}+ 
X^{[4]}(Y^{[4]}Z^{[3]}+Z^{[7]})+WZ^{[10]},$ then $\mu(I) = 9$ and
$\mu(J) = 5$. In the above examples, we have $\deg_X F < \deg G_1$.
While the first one is generic, the second one is not. This shows that
our characterization is not valid without the given hypotheses.
\end{example} 

The example below shows that the converse of Theorem \ref{tm2} is not
true.

\begin{example}\label{ex1thm2}
It can be seen that
$\ann(Y^{[5]}Z^{[12]}+ Y^{[4]}Z^{[13]} + X^{[8]}(Y^{[4]}Z^{[5]} +
Y^{[3]}Z^{[6]} + Y^{[2]}Z^{[7]} + YZ^{[8]}) + X^{[9]}(Y^{[4]}Z^{[4]} +
Y^{[3]}Z^{[5}) + X^{[10]}(Y^{[5]}Z^{[2]} + Y^{[4]}Z^{[3]} +
Y^{[3]}Z^{[4]} + Z^{[7]}) + WZ^{[16]})$
is generic whereas 
$\ann(Y^{[5]}Z^{[12]}+ Y^{[4]}Z^{[13]}+ X^{[7]}(Y^{[4]}Z^{[6]}+
Y^{[3]}Z^{[7]}+ Y^{[2]}Z^{[8]}+ YZ^{[9]})+ X^{[9]}(Y^{[4]}Z^{[4]}+
Y^{[3]}Z^{[5]})+ X^{[11]}(Y^{[5]}Z+ Y^{[4]}Z^{[2]}+ Y^{[3]}Z^{[3]}+
Z^{[6]})+ WZ^{[16]})$ is not. Note that in both cases $j-a_1 < a_1$,
$\deg_Y G_i < \deg_Y G_0$ and $\deg_Y G_3 = \deg_Y G_0 = 5$.
\end{example} 
The example below shows that we cannot remove the conditions $\deg _Y = G_i < \deg_Y = G_0, 1\le i\le n-1$ in 
Theorem \ref{tm2}.  In $F$, all conditions except $\deg _Y G_1 = 6 > \deg _Y G_0=  5$ are satisfied and $\ann F$ is generic and the conclusion fails where as in $G$ all conditions except $\deg_Y = G_0 \neq \deg _Y = G_n$ .  Hence by the theorem \ref{tm2}, $\ann G$ cannot be generic. 
\begin{example} \label{ex2thm2}  
Let $F= Y^{[5]}Z^{[12]}+ Y^{[4]}Z^{[13]}+
X^{[9]}(Y^{[4]}Z^{[4]})+ X^{[10]}(Y^{[6]}Z+ Y^{[2]}Z^{[5]}+
Y^{[3]}Z^{[4]}+ Z^{[7]})+ X^{[12]}(YZ^{[4]}+ Z^{[5]})+ WZ^{[16]}$.
In this case, $6 = r_{21}>p = 5$, $ \mu(I)=15 $ and $ \mu(J) = 11$,
which implies that $I$ is generic. 
On
the other hand, if $G= Y^{[5]}Z^{[12]}+ Y^{[4]}Z^{[13]}+
X^{[9]}(Y^{[4]}Z^{[4]})+ X^{[10]}(Y^{[6]}Z+ Y^{[2]}Z^{[5]}+
Y^{[3]}Z^{[4]}+ Z^{[7]})+ WZ^{[16]}$, then  $\mu(I)=15 $ and $ \mu(J)
= 12$ as predicted by Theorem \ref{tm2}.
\end{example}

\end{document}